\documentclass[12pt]{article}
\usepackage[margin=1in]{geometry}
\usepackage{color}
\usepackage{amsmath,amsthm,amssymb,mathtools,bbm,mathrsfs}
\usepackage[round,longnamesfirst]{natbib}
\usepackage{array}
\usepackage{enumitem}
\usepackage{textcase}

\renewcommand{\cite}{\citet*}
\usepackage[breaklinks=true,hidelinks]{hyperref}

\numberwithin{equation}{section}
\allowdisplaybreaks[4]

\newtheoremstyle{plain2}{\topsep}{2\topsep}{\itshape}
{0pt}{\bfseries}{.}{.5em}{}
\newtheoremstyle{definition2}{\topsep}{2\topsep}{}
{0pt}{\bfseries}{.}{.5em}{}

\theoremstyle{plain2}
\newtheorem{theorem}{Theorem}[section]
\newtheorem{proposition}[theorem]{Proposition}
\newtheorem{lemma}[theorem]{Lemma}
\newtheorem{corollary}[theorem]{Corollary}

\theoremstyle{definition2}

\makeatletter

\renewcommand{\cite}{\citet}

\def\^#1{\ifmmode {\mathaccent"705E #1} \else {\accent94 #1} \fi}
\def\~#1{\ifmmode {\mathaccent"707E #1} \else {\accent"7E #1} \fi}

\def\*#1{#1^\ast}
\edef\-#1{\noexpand\ifmmode {\noexpand\bar{#1}} \noexpand\else \-#1\noexpand\fi}
\def\>#1{\vec{#1}}
\def\.#1{\dot{#1}}

\def\atop{\@@atop}
\def\%#1{\mathcal{#1}}

\let\original@left\left
\let\original@right\right
\renewcommand{\left}{\mathopen{}\mathclose\bgroup\original@left}
\renewcommand{\right}{\aftergroup\egroup\original@right}

\renewcommand{\phi}{\varphi}
\newcommand{\eps}{\varepsilon}

\newcommand{\eq}{\eqref}

\def\ER{Erd\H{o}s--R\'enyi}

\newcommand{\Po}{{\rm Po}}

\newcommand{\Bi}{\mathop{\mathrm{Bi}}}

\newcommand{\Var}{\mathop{\mathrm{Var}}}

\newcommand{\sgn}{\mathop{\mathrm{sgn}}}
\newcommand{\law}{\mathscr{L}}

\def\be#1{\begin{equation*}#1\end{equation*}}
\def\ben#1{\begin{equation}#1\end{equation}}
\def\bes#1{\begin{equation*}\begin{split}#1\end{split}\end{equation*}}
\def\besn#1{\begin{equation}\begin{split}#1\end{split}\end{equation}}

\def\ba#1{\begin{align*}#1\end{align*}}
\def\ban#1{\begin{align}#1\end{align}}

\def\given{\typeout{Command 'given' should only be used within bracket command}}
\newcounter{@bracketlevel}
\def\@bracketfactory#1#2#3#4#5#6{
\expandafter\def\csname#1\endcsname##1{%
\addtocounter{@bracketlevel}{1}%
\global\expandafter\let\csname @middummy\alph{@bracketlevel}\endcsname\given%
\global\def\given{\mskip#5\csname#4\endcsname\vert\mskip#6}\csname#4l\endcsname#2##1\csname#4r\endcsname#3%
\global\expandafter\let\expandafter\given\csname @middummy\alph{@bracketlevel}\endcsname
\addtocounter{@bracketlevel}{-1}}%
}
\def\bracketfactory#1#2#3{%
\@bracketfactory{#1}{#2}{#3}{relax}{1mu plus 0.25mu minus 0.25mu}{0.6mu plus 0.15mu minus 0.15mu}
\@bracketfactory{b#1}{#2}{#3}{big}{1mu plus 0.25mu minus 0.25mu}{0.6mu plus 0.15mu minus 0.15mu}
\@bracketfactory{bb#1}{#2}{#3}{Big}{2.4mu plus 0.8mu minus 0.8mu}{1.8mu plus 0.6mu minus 0.6mu}
\@bracketfactory{bbb#1}{#2}{#3}{bigg}{3.2mu plus 1mu minus 1mu}{2.4mu plus 0.75mu minus 0.75mu}
\@bracketfactory{bbbb#1}{#2}{#3}{Bigg}{4mu plus 1mu minus 1mu}{3mu plus 0.75mu minus 0.75mu}
}
\bracketfactory{clc}{\lbrace}{\rbrace}
\bracketfactory{clr}{(}{)}
\bracketfactory{cls}{[}{]}
\bracketfactory{abs}{\lvert}{\rvert}
\bracketfactory{norm}{\Vert}{\Vert}
\bracketfactory{floor}{\lfloor}{\rfloor}
\bracketfactory{ceil}{\lceil}{\rceil}
\bracketfactory{angle}{\langle}{\rangle}

\newcounter{ctr}\loop\stepcounter{ctr}\edef\X{\@Alph\c@ctr}%
	\expandafter\edef\csname s\X\endcsname{\noexpand\mathscr{\X}}
	\expandafter\edef\csname c\X\endcsname{\noexpand\mathcal{\X}}
	\expandafter\edef\csname b\X\endcsname{\noexpand\boldsymbol{\X}}
	\expandafter\edef\csname I\X\endcsname{\noexpand\mathbbm{\X}}
\ifnum\thectr<26\repeat

\newcount\minute
\newcount\hour
\newcount\hourMins
\def\now{%
\minute=\time%
\hour=\time \divide \hour by 60%
\hourMins=\hour \multiply\hourMins by 60%
\advance\minute by -\hourMins%
\zeroPadTwo{\the\hour}:\zeroPadTwo{\the\minute}%
}
\def\zeroPadTwo#1{\ifnum #1<10 0\fi#1}

\renewcommand\section{\@startsection {section}{1}{\z@}%
{-3.5ex \@plus -1ex \@minus -.2ex}%
{1.3ex \@plus.2ex}%sm
{\center\small\sc\mathversion{bold}\MakeTextUppercase}}

\def\subsection#1{\@startsection {subsection}{2}{0pt}%
{-3.5ex \@plus -1ex \@minus -.2ex}%
{1ex \@plus.2ex}%
{\bf\mathversion{bold}}{#1}}

\def\subsubsection#1{\@startsection{subsubsection}{3}{0pt}%
{\medskipamount}%
{-10pt}%
{\normalsize\itshape}{\kern-2.2ex. #1.}}

\def\blfootnote{\xdef\@thefnmark{}\@footnotetext}

\makeatother

\def\sp#1{^{(#1)}}

\def\ff{{\mathcal F}}
\def\giv{\,|\,}
\def\a{\alpha}
\def\b{\beta}
\def\s{\sigma}
\def\d{\delta}

\def\law{{\mathcal L}}

\def\cE{{\mathcal E}}
\def\ignore#1{}

\def\hB{{\widehat B}}

\def\half{\tfrac12}

\def\l{\lambda}
\def\L{\Lambda}
\def\Eq{\ =\ }
\def\Le{\ \le\ }
\def\Def{\ :=\ }
\def\CCC{A}
\def\dbw{d_{\mathrm{BW}}}
\def\hbw{{\cH}_{\mathrm{BW}}}
\def\iid{independent and identically distributed}

\def\hN{{\widehat N}}
\def\hX{{\widehat X}}
\def\hNN{{\widehat \cN}}
\def\ui{^{(1)}}
\def\rrr{q}
\def\Ref{\eq}
\def\url#1{{\tt #1}}
\def\nin{\noindent}

\begin{document}

\title{\sc\bf\large\MakeUppercase{Central moment inequalities using Stein's method}}
\author{\sc A. D. Barbour\thanks{Universit\"at Z\"urich; {\tt a.d.barbour@math.uzh.ch}}, 
Nathan Ross\thanks{University of Melbourne; {\tt nathan.ross@unimelb.edu.au}} and 
Yuting Wen\thanks{University of Melbourne; {\tt yutingyw@gmail.com}}}
%\date{\it  }
\maketitle

\begin{abstract}  
We derive explicit central moment inequalities for random variables that admit a Stein coupling, 
such as exchangeable pairs, size--bias couplings or local dependence, among others. The bounds are in 
terms of  moments (not necessarily central) of variables in the Stein coupling, which are typically 
local in some sense, and therefore easier to bound. In cases where the Stein couplings 
have the kind of behaviour leading to good normal approximation, the central moments are closely bounded 
by those of a normal. We show how the bounds can be used to produce concentration inequalities, and compare 
them to those existing in related settings. Finally, we illustrate the power of the theory by bounding the 
central moments of sums of neighbourhood statistics in sparse \ER\ random graphs. 
\end{abstract}

%\noindent\textbf{Keywords: } 

\section{Introduction}
Concentration inequalities are useful and powerful tools 
for estimating probabilities when exact computation is not possible. They have found 
important application in modern statistics \citep{Massart2007}. 
Obtaining such inequalities has been an active area of probability for decades; for example, see \cite{Boucheron2013} 
and the references there.  The standard method for deriving 
strong concentration inequalities requires bounds on the moment generating function. For sums of independent random 
variables, this leads to Hoeffding/Bennett/Chernoff inequalities. Such results generalize to martingales in various 
ways \citep{McDiarmid1998}, and to
%. The recent literature is too vast to responsibly survey here, but there has been significant focus on 
functions of independent random variables \citep{Boucheron2003}. The research closest to 
this paper is concerned with showing concentration inequalities for random variables admitting various coupling 
constructions related to Stein's method. The first results and key ideas are due to \cite{Chatterjee2007} and 
\cite{Chatterjee2010}, who worked under the assumption that it is possible to construct an exchangeable pair that 
is marginally distributed as the variable of interest, and has certain conditional moment boundedness properties. 
The ideas were extended to bounded size--bias couplings by \cite{Ghosh2011a, Ghosh2011b}, and then further to unbounded size--bias couplings satisfying a certain bounded in conditional probability assumption, by \cite{Cook2018}. The exchangeable pair results have been 
generalized to the matrix setting by \cite{Mackey2014}.

When it is not possible to bound the moment generating function, either due to its non-existence or to the complexity of 
the distribution of interest, another approach, that gives (sometimes only slightly) weaker  concentration inequalities, 
is to bound moments of the variable of interest. 
For sums of independent random variables, there are a number of Rosenthal/Pinelis/Burkholder/Davis/Gundy inequalities. 
For functions of independent random variables, this program has been developed in \cite{Boucheron2005}. 
Results for exchangeable pairs with additional structure are developed in \cite[Theorem~1.5(iii)]{Chatterjee2007}, 
with matrix versions in \cite{Paulin2016}.

While Stein's method has found huge success in bounding the error made in distributional approximation 
(see \cite{Barbour1992}, \cite{Chen2011} and \cite{Chatterjee2014}), the techniques for concentration inequalities lag 
behind. This is essentially because controlling the moment generating function requires additional restrictive 
assumptions, such as boundedness, on the couplings that have found such great success in Stein's 
method for distributional approximation. In this paper, we develop explicit general bounds on the central moments of a 
random variable of interest,  which are expressed in terms of the moments of the key variables appearing in 
a \emph{Stein coupling}.   In the context of sums of independent random variables, the analogues of these variables 
are the individual summands, and it is usual there to try to express a bound on the moments of the sum in
terms of those of the summands.  The Stein coupling formulation allows sums of random variables
with a wide range of dependence structures to be treated in a unified way;  the key variables may then
be different, but their moments still provide accessible quantities for expressing a bound.
We demonstrate the applicability of our approach in the
settings of sums of independent random variables, size--bias couplings and local dependence, 
and we compare our bounds to those previously derived.  
However, the main advantage of our approach is that it can be used to obtain concentration in applications 
where other methods have not been developed.  We illustrate this by showing concentration for 
the distribution of generic neighbourhood statistics in sparse \ER\ random graphs. 

In the next section, we introduce the setting for the couplings we consider and state our moment inequalities. 
Section~\ref{sec:concin} contains a discussion on how these moment inequalities translate to concentration 
inequalities and Section~\ref{sec:app} contains applications. The proofs are in Section~\ref{sec:proofs}.

\section{Stein couplings and moment inequalities}

The ordered collection of random variables $(W,W',G,R)$ with $\IE W=\mu$ are said to form an 
\emph{approximate Stein coupling} if
\ben{\label{eq:stncoup}
   \IE[G (f(W')-f(W))] \Eq \IE[(W-\mu)f(W)]+\IE [Rf(W)],
}
for all $f$ such that the expectations exist. If the remainder~$R$ is identically zero,  
the triple $(W,W',G)$ is called a \emph{Stein coupling}. 
Standard examples of Stein couplings are discussed below in Section~\ref{sec:app}.
These couplings are useful for establishing the error in approximating the distribution of~$W$ by
a normal distribution, using Stein's method;  see \cite{Chen2010}.
 With  $\s^2 := \Var W$, Stein's method shows that, if 
\[
     |\IE[(W-\mu)f(W) - \s^2 f'(W)]| \Le \eps(\s^2\|f'\|_\infty + \s^3 \|f''\|_\infty)
\]
for all~$f$ in a suitable class of functions with two bounded derivatives, then $\dbw(\law(\s^{-1}(W-\mu)),\cN(0,1)) = O(\eps)$ as $\eps\to0$.
Here, the bounded Wasserstein distance~$\dbw$ between probability measures $P$ and~$Q$ on~$\IR$ is defined
by
\[
      \dbw(P,Q) \Def \sup_{h \in \hbw}\bbabs{\int h\, dP - \int h\,dQ},
\]
where $\hbw := \{h\colon\,\IR \to \IR \colon\, \|h\|_\infty \le 1, \|h'\|_\infty \le 1\}$;  it metrizes
weak convergence.
Given a Stein coupling, and writing $D := W'-W$, it is immediate that
\bes{
   \lefteqn{ \IE[(W-\mu)f(W) - \s^2 f'(W)] \Eq \IE[G (f(W')-f(W))] - \s^2\IE f'(W) }\\
       &\qquad\qquad\qquad\qquad\Eq \IE[(GD-\s^2)f'(W)] + \IE[G(f(W+D) - f(W) - Df'(W))].
}
Hence, to establish the accuracy of normal approximation, it suffices to show, for instance, 
that $\IE|\IE\cls{GD \giv W} - \s^2| \le \eps \s^2$
and that $\IE|GD^2| \le \eps\s^3$, for suitable choice of~$\eps$.

In this paper, we show that the existence of approximate Stein couplings also yields
bounds for the even central moments of~$W$. To some extent, the bounds cover cases where
the distribution of~$W$ is not very close to being normal. To state our first theorem, 
we define $\|X\|_r := \clc{\IE\abs{X^r}}^{1/r}$ for
any random variable~$X$ and $r\in\IN:=\{1,2,\ldots\}$. Throughout the paper, we write
$\mu:=\IE W$ and $\sigma^2:=\Var(W)$.

\begin{theorem}\label{thm:stnmomineq}
Let $(W,W',G,R)$ be an approximate Stein coupling such that, for some constant $\eps$ such that $0\leq \eps<1$ and 
for some random variable $T\geq 0$,
\ben{\label{eq:remcond}
       \babs{\IE\cls{R|W}} \Le \eps\abs{W-\mu} + T.
}
Let $k\in\IN$ and suppose that $\|G\|_{2k} \le \CCC_k=:\CCC$ and that $\|D\|_{2k}\le B_k=:B$. Then
\bes{
    \|W-\mu\|_{2k} &\Le \frac{\CCC}{1-\eps } 
     \left[\left(1+\sqrt{\frac{B(1-\eps)}{\CCC(2k-1)}}\right)^{2k-1}-1\right] 
             +\frac{\norm{T}_{2k}}{1-\eps} \\
   &\Le \sqrt{\frac{(2k-1) \CCC B}{1-\eps}} \exp\left\{ \sqrt{\frac{B(1-\eps)(2k-1)}{\CCC}}\right\}
             +\frac{\norm{T}_{2k}}{1-\eps}\,.
}
\end{theorem}

\medskip
The statement of the theorem looks at first sight rather complicated.  
Before discussing it, we state a related result with a slightly simpler bound
under the additional assumption
that, for each $r\in\IN$ of interest, $\IE\clc{|W'-\mu|^{r}} \leq \IE\clc{|W-\mu|^{r}}$, together with an explicit
statement about the smallness of~$\|T\|_{r}$.
%, there is a simpler bound,
%given in the next theorem.  
The extra assumption is satisfied for exchangeable pairs, when $\law(W') = \law(W)$, as well
as for the Stein coupling that we shall use for sums of independent random variables with mean zero.  
Note that, in this theorem, odd powers of $|W-\mu|$ are also allowed.
 
\begin{theorem}\label{thm:stnmomineq0}
Let $(W,W',G,R)$ be an approximate Stein coupling such that~\eq{eq:remcond} is satisfied.
%Assume also that, for each $l\in\IN$, $\IE\clc{|W'-\mu|^{l}} \leq \IE\clc{|W-\mu|^{l}}$. 
Then, for any $r\in\IN$ such that $\IE\clc{|W'-\mu|^{r}} \leq \IE\clc{|W-\mu|^{r}}$
and that $\s^{-1}\|T\|_{r} \le \eps'$ for some~$\eps' < 1 - \eps$,
it follows that
\bes{
    \|W-\mu\|_{r} 
   &\Le \sqrt{\frac{2(r-1) \|G\|_r \|D\|_r}{1-\eps-\eps'}}\,.
}
\end{theorem}

\medskip
Theorems \ref{thm:stnmomineq} and~\ref{thm:stnmomineq0} provide bounds for central moments of~$W$, expressed simply in
terms of the moments of the random variables $G$ and~$D$, which, in practical circumstances,
are easier to handle than~$W$ itself.  However, it is not at first sight clear how good the bounds 
are likely to be.  To get some idea of this, consider the case of a Stein coupling, when~$R=0$.
In this case, taking $f(w) = w-\mu$ in~\eq{eq:stncoup}, it follows that $\s^2  = \IE(GD)$. Thus,
using the Cauchy--Schwarz and H\"older inequalities, we have
\[
     \s \Eq \clc{\IE(GD)}^{1/2} \Le \clc{\|G\|_2\|D\|_2}^{1/2} \Le \clc{\|G\|_{2k}\|D\|_{2k}}^{1/2}
            \Le \clc{\CCC B}^{1/2}.
\]
Hence Theorems \ref{thm:stnmomineq} and~\ref{thm:stnmomineq0} show that $\|W-\mu\|_{2k}$ is sandwiched between~$\s$ 
and a multiple of
$\clc{\CCC B}^{1/2}$, which is the upper bound for~$\s$ obtained by replacing the $2$-norms of $G$ and~$D$ 
in the Cauchy--Schwarz inequality by
their $2k$-norms.  In many applications concerning asymptotics as the size~$n$ of a problem increases, 
the distribution of~$D$ remains more or less constant, while~$\s^2$
grows like~$n$.  Since $\IE(GD) = \s^2$, the distribution of~$n^{-1}G$ 
also remains more or less constant.  Hence $B$ can typically be chosen more or less constant in~$n$, whereas
$\CCC$ is proportional to~$n$, implying that $B/\CCC = O(n^{-1})$.  Thus the exponential factor in the second
inequality of Theorem~\ref{thm:stnmomineq} is close to~$1$ unless~$k$ is very large, making the constant
multiplying $\clc{\CCC B}^{1/2}$ smaller than that in Theorem~\ref{thm:stnmomineq0}.  

For example, for a sum $W := \sum_{i=1}^n X_i$ of \iid\ random variables
with zero mean, there is a Stein coupling $(W,W-X_I,-nX_I)$, where~$I$ denotes a random variable
with the uniform distribution on $\{1,2,\ldots,n\}$ that is independent of $(X_i)_{i=1}^n$, so that, 
in particular, $\eps=0$.  Then $\s^2 =  n\IE(X_1^2)$ and~$G$ have factors of~$n$, 
%in their definition, 
whereas $\IE(D^2) = \IE(X_1^2)$ is constant in~$n$.  In particular,
we can take $\CCC$ and~$B$ such that
$\CCC B = n\|X_1\|_{2k}^2$, as compared to $\s^2 = n\|X_1\|_2^2$, and, for fixed~$k$, the bound
on $\|W-\mu\|_{2k}$ in Theorem~\ref{thm:stnmomineq} is close to $\sqrt{n(2k-1)}\|X_1\|_{2k}$. 
This Stein coupling satisfies the extra condition of Theorem~\ref{thm:stnmomineq0}, giving
the bound $\sqrt{2n(2k-1)}\|X_1\|_{2k}$.
Note that, by considering $X_i \sim 2I-1$, where $I \sim {\rm Bernoulli}(1/2)$, the factor of $\sqrt{2k-1}$ 
is seen to be inevitable, being the rate of the growth of $\norm{N}_{2k}$ for $N$ standard normal; see Lemma~\ref{lem:asymnormoms}.

Theorems~\ref{thm:stnmomineq} and~\ref{thm:stnmomineq0} are our most useful results for difficult applications, 
such as Theorem~\ref{thm:nhoodfuncmombd} below.  However, we also show that sharper bounds can be obtained, if 
Stein's method of normal approximation applies to the variable under consideration.
The next theorem shows that, if this is the case,
the product $\CCC B$ in the leading term can be replaced by~$\s^2$.  The extra condition, given
in~\eq{ADB-0},  would be expected to be satisfied with~$\eps_3$ small, if normal approximation
were good.  Thus, if normal approximation is good enough, the even central moments of~$W$
cannot be too much bigger than those of the corresponding normal distribution.

\begin{theorem}\label{thm:stnmomineq2}
With the assumptions and notation of Theorem~\ref{thm:stnmomineq}, suppose that now,
for non-negative random variables $T_1$ and~$T_2$ and for non-negative $\eps_i$,
$1\le i\le 3$, we have
\ben{\label{eq:remcond2}
       \babs{\IE\cls{R|W}} \Le \eps_1\abs{W-\mu} + T_1,
}
where $\s^{-1}\|T_1\|_{2k} \le \eps_2$, and 
\ben{\label{ADB-0}
    |\IE\cls{GD \giv W} - \s^2| \Le \eps_3 (W-\mu)^2 + T_2.
}
Then, if $\eps_1+\eps_2+(2k-1)\eps_3 < 1$, it follows that
\bes{
   \|W-\mu\|_{2k} &\Le \frac{\s \sqrt{2k-1}}{\sqrt{1-\eps_1-\eps_2-(2k-1)\eps_3}}
        \bbbclr{1 + \frac{(k-1)\CCC B^2}{\s^3\sqrt{2k-1}} \exp\bbclc{\frac{B\sqrt{2k-1}}{\s}} 
               + \frac{\|T_2\|_k}{\s^2}}^{1/2}.
}
% where~$N$ denotes a standard normal random variable.
\end{theorem}

\nin Note that it is always possible to take $\eps_3=0$ and $T_2 := |\IE\cls{GD \giv W} - \s^2|$.
Note also that 
the conditions of Theorem~\ref{thm:stnmomineq2} are similar to others used in the Stein's method literature; 
for example, see \cite[Theorem~3.1]{Chen2013a}, \cite[(2.4)]{Barbour2019}, and \cite[Theorem~3.11]{Rollin2007}, 
covering applications such as the anti-voter model and the number of isolated vertices in a sparse 
\ER\ random graph. 

\medskip
Supposing that $\s^{-2}\|T_2\|_k \le \eps_4$ with~$\eps_4$ small enough,
we have a further refinement, showing that the even central moments of~$W$ are bounded by a
quantity which, under typical asymptotics, is equivalent to the corresponding
normal moment.  To state the theorem, we define  
\ben{\label{ADB-h_k}
    h_k \Def 2^{-1/2}e^{5/2} \s^{-3}\CCC B^2 \sqrt{k-1}.
}

\begin{theorem}\label{thm:stnmomineq3}
Under the assumptions of Theorem~\ref{thm:stnmomineq2},
and assuming that $\s^{-2}\|T_2\|_k \le \eps_4$, suppose that
\[
    E \Def \eps_1 + \eps_2 + (2k-1)(\eps_3 + \eps_4) \ <\ 1 - h_k,
\]
and that $\s \ge B\sqrt{e(2k-1)}$. Then
\bes{
   \|\s^{-1}(W-\mu)\|_{2k} &\Le  \|N\|_{2k}/\sqrt{1-E-h_k}.        
}
where~$N$ denotes a standard normal random variable.
\end{theorem}

We end this section with a lemma that is used in our proofs, but is also important in interpreting the sharpness 
of central moment bounds.

\begin{lemma}\label{lem:asymnormoms}
For~$N$ a standard normal random variable,
define
\be{
   c_1(k) \Def \clc{\sqrt{2k-1} / \|N\|_{2k}}.
}
Then $c_1$ is increasing in~$k$, with $c_1(1) = 1$ and $c_1(\infty) = \sqrt{e}$.
\end{lemma}

\begin{proof}
First, $\log c_1(k) < 1/2$
follows using the usual bounds for the error in Stirling's formula.  Then, using the expansion of 
$(x + 1/2) \log(1 + 1/x)$ for $x > 1$ in the first inequality, we find that
\ben{\label{eq:ckin}
     \frac{c_1(k+1)^{2(k+1)}}{c_1(k)^{2k}} \Eq \bbclr{1 + \frac{2}{2k-1}}^k\ >\ e\ >\ c_1(k)^2, 
}
which implies that $c_1(k+1) > c_1(k)$, and then also that $\lim_{k\to\infty} c_1(k) = \sqrt{e}$.
\end{proof}

\section{Concentration inequalities from central moment bounds}\label{sec:concin}
The bounds derived above can be used with Markov's inequality to show that the distribution of
a random variable~$W$ is concentrated about its mean, by starting from  
\ben{\label{eq:Mineq}
      \IP[|W-\mu| > t] \Le t^{-2k}\|W-\mu\|_{2k}^{2k},
}
and choosing~$k$ carefully.
Suppose that $(W_n,\,n\ge1)$ is a sequence of random variables with means $\mu_n$ and variances
$\s_n^2 \asymp n$.  
Then one weak form of concentration is to say that~$W_n$ is concentrated
about its mean~$\mu_n$ on the scale $(d_n)_{n\ge1}$ if, for any $r \ge 1$, there exist $K(r)$
and~$c(r)$ such that
\[
    \IP[|W_n - \mu_n| > c(r)d_n] \Le K(r)n^{-r},
\]
uniformly in~$n$.
Under the assumptions of Theorem~\ref{thm:stnmomineq}, and for an exact Stein coupling
with $\norm{n^{-1}G}_{2k}\leq \alpha_{n,k}$ and $\norm{D}_{2k}\leq \beta_{n,k}$, 
we have, 
\ben{\label{ADB-new-0}
   \IP[|W_n - \mu_n| > t_n] \Le t_n^{-2k}\bbclc{n(2k-1)\alpha_{n,k}\beta_{n,k}}^k
            \exp\bbbclc{\frac{2k}{\sqrt n}\,\sqrt{\frac{(2k-1)\beta_{n,k}}{\alpha_{n,k}}}}.
}
Thus, taking  $k = k_n := \log n$ and
\ben{\label{ADB-new-1}
     d_n \Def \sqrt{n(2k_n-1)\alpha_{n,k_n}\beta_{n,k_n}} 
                     \exp\bbbclc{\frac1{\sqrt n}\,\sqrt{\frac{(2k_n-1)\beta_{n,k_n}}{\alpha_{n,k_n}}}},
}
we have
\be{%\label{ADB-new-2}
     \IP[|W_n - \mu_n| > cd_n] \Le c^{-2\log n} \Eq n^{-2\log c},
}
which can be made smaller than~$n^{-r}$, for any given~$r$,
by choosing $c = c(r)$ large enough.  Thus, in this sense, Theorem~\ref{thm:stnmomineq} shows that~$W_n$
is concentrated around~$\mu_n$ on a scale~$(d_n)_{n\ge1}$, where~$d_n$ is as in~\eq{ADB-new-1}.

As remarked earlier, it is often the case in such asymptotics that the distributions
of $D$ and~$n^{-1}G$ are more or less constant in~$n$, in the sense that their tails are
uniformly bounded in~$n$.  However, for any~$n$, the norms $\|D\|_k$ and~$\|n^{-1}G\|_k$  
grow to infinity as~$k$ increases, unless the random variables themselves are uniformly bounded. The following lemma
is useful in determining how fast the norms grow with~$k$;  its proof is straightforward, 
by using, for example, saddle point methods.

\begin{lemma}\label{moments-from-tails}
 Let~$X$ be a random variable such that, for some $a,b,c > 0$, 
\ben{\label{ADB-new-3}
          \IP[|X| > x] \Le c e^{-bx^a} \quad\mbox{for all}~x > 0.
}
Then, for $k \ge 1$,
\be{%\label{ADB-new-3a}
     \IE |X|^{2k} \Le \frac{2kc}{a}\,\bbclr{b^{-1/a}}^{2k} \Gamma(2k/a).
}
If, instead, $a>1$ and
\ben{\label{ADB-new-4}
     \IP[|X| > x] \le c e^{-b(\log(x+1))^a}  \quad\mbox{for all}~x > 0,
} 
then, for $k \ge 1$,
\be{%\label{ADB-new-4a}
    \IE |X|^{2k} \Le c C_1(a,b) k^{a/(2(a-1))} \exp\{C_2(a,b) k^{a/(a-1)}\},
}
for suitable constants $C_1(a,b)$ and~$C_2(a,b)$.
\end{lemma}

Hence the quantity~$d_n$ above is of order $O(\sqrt{n\log n})$ if both $D$ and~$n^{-1}G$ are a.s.\ 
bounded for all~$n$ by the same constants $x_1$ and~$x_2$.  If both $D$ and~$n^{-1}G$ have tails bounded 
as in~\eq{ADB-new-3},
uniformly for all~$n$, then $d_n = O((\log n)^{1/a + 1/2}\sqrt n)$;  if they have tails bounded as in~\eq{ADB-new-4},
uniformly for all~$n$, and if $a > 2$, then $d_n = O(n^{1/2+\d})$ for any $\d > 0$.

The classical large deviation bounds, such as the Chernoff bounds, deliver much smaller bounds
for the probabilities of large deviations than those required for the weaker form of concentration
discussed above.  We now show that our moment bounds can also deliver analogous results, if values
of~$k_n$ larger than~$\log n$ are taken.

Under the conditions of Theorem~\ref{thm:stnmomineq}, we can invoke~\eq{ADB-new-0}, and choose $k = k_n(t)$
to make the principal factor $t^{-2k}\bbclc{n(2k-1)\|n^{-1}G\|_{2k}\|D\|_{2k}}^k$ small, for fixed choice of~$t$. 
In particular we have the following corollary for a.s.\ bounded $n^{-1}G$ and~$D$.

\begin{corollary}\label{ADB-bounded}
 Under the assumptions of Theorem~\ref{thm:stnmomineq}, for an exact Stein coupling
such that $|D| \le x_1$ and $|n^{-1}G| \le x_2$ a.s.\ for all~$n$, we have,
for any $t\ge \sqrt{2 nx_1x_2e}$,
\besn{\label{ADB-new-5a}
   \IP[|W_n - \mu_n| > t] % &\Le e\exp\bbbclc{-\frac{t^2}{2nx_1x_2e}
%           \bbbclr{1 - \frac{2}{\sqrt n}\,\sqrt{\frac{(2k-1)x_1}{\|n^{-1}G\|_2}}}} \\
      &\Le       e\exp\bbbclc{-\frac{t^2}{2nx_1x_2e}
           \bbbclr{1 - \frac{2t}{ x_2  n \sqrt e}}}.
}
\end{corollary}

\nin The corollary gives good bounds as long as $t \ll n$.
However, the factor~$e$ in the denominator makes the
exponent smaller than that in the Chernoff bound;  we return to this later.

If we only have control over the tails of $|D|$ and~$|n^{-1}G|$, we can obtain the following.

\begin{corollary}\label{ADB-Weibull-tails}
 Under the assumptions of Theorem~\ref{thm:stnmomineq}, for an exact Stein coupling
such that both $|D|$ and~$|n^{-1}G|$ have tails bounded as in~\eq{ADB-new-3}, then for $t^2 \ge ne2^{(a+2)/a}(ab)^{-2/a}$, we have
\bes{%\label{ADB-new-6}
   \IP[\abs{W_n - \mu_n} > t] 
      &\Le   c\sqrt{2\pi}\, \bbbclc{\frac{2k_t}{a}+\frac{1}{2}}^{1/2}  e^{-k_t(a+2)/a}\,\exp\{n^{-1/2}(2k_t)^{3/2}\},
     }
where $k_t := \bfloor{\half\bclr{t^2/(ne)}^{a/(a+2)}(ab)^{2/(a+2)}}$.
\end{corollary}

\nin It is clear that under the hypotheses of the corollary, there are constants $K_1, K_2$ and~$K_3$ depending on $a$ and $b$ such that
\bes{%\label{ADB-new-6}
   \IP[\abs{W_n - \mu_n} > t] 
      &\Le   c K_1 (t^2/n)^{a/(2(a+2))}\exp\bbclc{-K_2\bbclr{\frac{t^2}{n}}^{a/(a+2)}\bbclc{1 - K_3\frac{t}{ n}}}.
     }
     The {\it power\/} of~$t$ in the exponent is now no longer as large
as~$2$, though it approaches $2$ as~$a$ increases.  On the other hand, the bound still gives useful results,
provided that $t \ll n$.

If the conditions of Theorem~\ref{thm:stnmomineq3} are satisfied, better bounds can be obtained;
in particular, if we have an exact Stein coupling for which the $k$-norms of $T_2 := |\IE\cls{GD \giv W} - \s^2|$
can be shown to be suitably small.

\begin{corollary}\label{cor:normal}
Under the assumptions of Theorem~\ref{thm:stnmomineq3}, taking $k = k(y) := \lceil y^2/2 \rceil$, we have
\ben{\label{ADB-new-8}
    \IP[\s_n^{-1}|W_n - \mu_n| > y] \Le e\sqrt 2 \exp\bbclc{-\frac{y^2}2}
              (1 - E - h_{k})^{-k}.
}
\end{corollary}

For an exact Stein coupling, the quantities $\eps_1$ and~$\eps_2$ are zero.  In most 
applications they, and $\eps_3$ and~$\eps_4$, can be expected to depend on~$n$ as a power~$n^{-\a}$, for a suitable
$\a > 0$ (often $\a=1/2$).  However, $\eps_2$ and~$\eps_4$ also involve $k$-norms of the error
random variables $T_1$ and~$T_2$, and the quantity~$h_k$ also involves $k$-norms of $D$ and~$n^{-1}G$.  Assuming
that their tails are bounded as in~\eq{ADB-new-3}, these norms can
be dealt with much as above, resulting in powers of~$k$ as factors;  so long as $y \ll n^\b$ for a suitably 
small index~$\b$, the bound
above is then useful.  In particular, if $|D|$, $|n^{-1}G|$, $n^{-1/2}|T_1|$ and~$n^{-1}|T_2|$ are uniformly bounded, then 
good bounds are obtained for $y \ll n^{1/4}$, equivalent to deviations of~$W_n$ from its mean of order $o(n^{3/4})$.
This is not as good a range as for the Chernoff bounds, but the main exponent is ideal.

\section{Applications}\label{sec:app}

Here we use Theorems \ref{thm:stnmomineq} -- \ref{thm:stnmomineq3} in some applications.

\subsection{Sums of independent random variables}
Let $W:=\sum_{i=1}^n X_i$, where $X_1,\ldots,X_n$
 are independent mean zero random variables with $\IE \abs{X_i}^{2k} <\infty$ for $1\leq i\leq n$.
Let $W'=W-X_I$, where $I$ is a uniform index from $\{1,\ldots, n\}$, and is independent of $(X_i)_{i=1}^n$.
%and $(X_1',\ldots, X_n')$ is an independent copy of $(X_1,\ldots,X_n)$. 
Then for $G:=-nX_I$, as mentioned above, it is easily checked  that
$(W, W', G)$ is a Stein coupling; see \cite{Chen2010}.
Note too that, for $D:=W'-W=-X_I$, we have $G=n D$ and so we can take $\CCC=n B$, with
\[
     B \Eq \|X_I\|_{2k} \Eq \bbclc{n^{-1}\sum_{i=1}^n \IE X_i^{2k}}^{1/2k}.
\]
Writing $\rho_k := \|X_I\|_{2k}/\|X_I\|_2$, the bound from Theorem~\ref{thm:stnmomineq} becomes
\bes{ \|\s^{-1} W\|_{2k}
    &\Le \frac{n\|X_I\|_{2k}}{\s}  \left[\left(1+\sqrt{\frac{1}{n(2k-1)}}\right)^{2k-1}-1\right] \\
    &\Le \rho_k \sqrt{(2k-1)}\exp\left\{ \sqrt{\frac{(2k-1)}{n}}\right\}.
}
%This bound should be compared to classical moment inequalities.

The theorems that give sharper bounds rely on establishing~\eq{ADB-0}.  For the coupling above,
writing $\s_i^2 := \IE X_i^2$, the simplest version is obtained by taking~$\eps_3=0$ and
\[
     T_2 \Def \IE\bbclc{\sum_{i=1}^n (X_i^2 - \s_i^2) \giv W},
\]
but this yields a result that is not as clean as the one following, that we derive by a slightly different argument.

 \begin{proposition}\label{ADB-independent}
For $W:=\sum_{i=1}^n X_i$, where $X_1,\ldots,X_n$
 are independent mean zero random variables with $\IE \abs{X_i}^{2k} <\infty$ for $1\leq i\leq n$ and an integer $k\geq2$, we have
\besn{\label{ADB-3a}
   \|\s^{-1}W\|_{2k} &\Le  \|N\|_{2k}/\sqrt{1-h_k'},        
}
where
\be{
h_k' \Def 5\sqrt{2e^3}\rho_k^3 \sqrt{\frac{k-1}n} \quad\mbox{and}\quad \rho_k \Def \|X_I\|_{2k}/\|X_I\|_2,
}
and $n$ is assumed to be large enough that $h_k'<1$.
\end{proposition}

\noindent To interpret the proposition, note that, if there is uniform control over the tails of the $X_i$, such as 
in Lemma~\ref{moments-from-tails}, then $h_k'\to0$ as $n\to\infty$ for any fixed $k$. Furthermore, even for 
tails as given by~\eq{ADB-new-4}, $h_k'\to0$ for $k$ growing as $(\alpha \log n)^\beta$, 
for $0 < \beta < a-1$.

\subsection{Local dependence} 
Let $W:=\sum_{i=1}^n X_i$, where the random variables $X_1,\ldots,X_n$ have mean zero, and where
$\IE \abs{X_i}^{2k} <\infty$, for $1\leq i\leq n$.
 Assume also, for each $i=1,\ldots,n$, that there are ``neighbourhoods''
$\mathcal{N}_i\subseteq \{1,\ldots, n\}$ such that~$X_i$ is independent of $W_i:=W-\sum_{j\in\mathcal{N}_i}X_j$.
Let $W':=W_I$, where $I$ is a uniform index from $\{1,\ldots, n\}$, independent of all else,
and set  $G:=-nX_I$. Then, it is easily checked that
$(W, W', G)$ is an exact Stein coupling; see also \cite[Construction~2A]{Chen2010}.
Here $D:=W'-W=-\sum_{j\in\mathcal{N}_I} X_j$ 
Thus we can apply the bound of Theorem~\ref{thm:stnmomineq} with 
\ba{
  A &\Eq n \left(\frac{1}{n}\sum_{i=1}^n\IE \babs{X_i}^{2k}\right)^{1/(2k)}; \qquad
  B \Eq \left(\frac{1}{n}\sum_{i=1}^n\IE \bbabs{\sum_{j\in\mathcal{N}_i} X_j}^{2k}\right)^{1/(2k)}.
}

If, for example,  for all $i=1,\ldots, n$, $d:=\max_i\abs{\mathcal{N}_i}$
and $x:=\max_i\norm{X_i}_{2k}$, then, using Minkowski's inequality, we could also take
\be{
 A \Eq n x, \,\,\, B \Eq d x,
}
and Theorem~\ref{thm:stnmomineq} implies that
\besn{\label{eq:locdepmombd} 
\norm{W}_{2k}
   &\Le    n x \left[\left(1+\sqrt{\frac{d}{n(2k-1)}}\right)^{2k-1}-1\right] \\
   &\Le  \sqrt{n} x\sqrt{d (2k-1)} \exp\left\{ \sqrt{\frac{d(2k-1)}{n}}\right\}.
}
Corollary~\ref{ADB-bounded} can also be invoked, giving a bound on large deviation probabilities of
\ben{\label{ADB-local-dev-bnd}
   \IP[\abs{W} > t] \Le e\exp\bbbclc{-\frac{t^2}{2endx^2}\bbbclr{1 - \frac{t}{nx}\sqrt{\frac6e}}}\,,
}
for all $t \ge \sqrt{2endx^2}$. This can be compared to \cite[Theorem~2.1]{Janson2004b}, which, under the stronger condition that $\abs{X_i}\leq x$, gives
\ben{\label{eq:janloc}
\IP[\abs{W} > t] \Le 2 \exp\bbbclc{-\frac{t^2}{2ndx^2}}\,.
}
This bound is better than~\eq{ADB-local-dev-bnd} because it does not have the factor of $e$ in the denominator of the exponent. 
However, in situations where the refined inequalities such as Corollary~\ref{cor:normal} apply, our results improve on~\eq{eq:janloc}.
Scan statistics furnish standard examples of this kind. If $Y_1,Y_2,\ldots,Y_n$ are independent random variables, and, 
for $j=1,\ldots, n-\ell+1$,  we define 
\be{
    X_j \Def \phi_j(Y_j,\ldots,Y_{j+\ell-1})-\IE \phi_j(Y_j,\ldots,Y_{j+\ell-1}),
}
where $\phi_j\colon \IR^\ell\to\IR$ is uniformly bounded for all~$j$,
then $W=\sum_{j=1}^{n-\ell+1} X_j$ satisfies the hypotheses above with $d=2\ell+1$.
The special case of head runs in Bernoulli trials is an example in the next section, where we show how to improve the basic inequality.

%Sharper bounds may in general be obtained by bounding~$B$ more carefully.

\subsection{Size--bias couplings}\label{sec:sbcomp}

For a random variable $W\geq0$ with $\IE W=\mu<\infty$,
we say that~$W^s$ has the size--bias distribution of~$W$ if
\be{
   \IE f(W^s) \Eq \frac{\IE [W f(W)]}{\mu},
}
for all functions $f$ such that the right hand side is well defined.
If $(W,W^s)$ is a coupling of a distribution with its size--bias 
distribution, and we define $W':=W^s$ and $G:=\mu$, 
then $(W,W',G)$ is a Stein coupling. 
Theorem~\ref{thm:stnmomineq}  easily applies in this setting with $D := W^s - W$, so that
we can take
\be{
    A \Eq \mu, \hspace{6mm} B \Eq \left(\IE\left[\abs{W^s-W}^{2k}\right]\right)^{1/(2k)}.
}

For {\it bounded\/} size--bias couplings, where $\abs{W^s-W}\leq c$ for some constant~$c$, then
$A=\mu$ and $B\leq c$, and in this case the bound~\eq{ADB-new-5a} becomes
\ben{\label{ADB-size-bias-bound}
  \IP[|W - \mu| > t] \Le e\exp\bbbclc{-\frac{t^2}{2e\mu c}\bbbclr{1 - \frac t\mu \sqrt{\frac6e}}}.
}
This is to be compared with the best known bounds under these hypotheses, those of \cite{Arratia2015}.  A very slightly 
weaker version of their bound is
\ben{\label{ADB-AB-bound}
    \IP[|W - \mu| > t] \Le  2 \exp\left\{-\frac{t^2}{2\mu c + 2ct/3}\right\}.
}
This is better than~\eq{ADB-size-bias-bound}, because it does not have the 
factor~$e$ in the denominator of the exponent.

However, in many circumstances, Theorem~\ref{thm:stnmomineq3} can also be applied,
and Corollary~\ref{cor:normal} then yields bounds for large deviation probabilities.
The quantity~$h_k$ defined in~\Ref{ADB-h_k} is already directly expressed in terms of $A$ and~$B$.
Because the size--bias coupling yields an exact Stein coupling, $\eps_1=\eps_2=0$.  We can also
take $\eps_3=0$, in which case
\[
      \s^{-2}T_2 \Eq \frac{\mu}{\s^2}\, \bbabs{\IE\clc{D \giv W} - \frac{\s^2}{\mu}}
                 \Eq \abs{ \IE\clc{D \giv W} - \IE D }/\IE D.
\]
Thus, in order to use~\eq{ADB-new-8}, all that we need in addition is a useful bound on
the $k$-norm  $\norm{\IE\clc{D \giv W} - \IE D}_k$.  In practice, it may be easier to bound
$\norm{\IE\clc{D \giv \ff} - \IE D}_k$ for a larger $\s$-field~$\ff$, with respect to which~$W$
is measurable.  Provided this norm is sufficiently small, Corollary~\ref{cor:normal}
yields bounds in which the leading term in the exponent is improved to $-t^2/(2\s^2)$,
though typically in a restricted range of~$t$;
note that the leading term in the exponent is now typically better than $-t^2/(2\mu c)$, 
since $\s^2 = \mu\IE D$, and $\IE D \le c$,
with equality only if $\law(a_1W) = \Po(a_2)$, for some $a_1,a_2 > 0$.

For example, \cite{Ghosh2011b} derive a bound of essentially the same form as~\Ref{ADB-AB-bound}
for the number of $m$-runs in~$n$ i.i.d.\ Bernoulli trials $\xi_1,\ldots,\xi_n$,
with common success probability~$p$, in which case $c = 2m-1$.  Letting $\ff := \clc{\xi_1,\ldots,\xi_n}$,
$\IE\clc{D \giv \ff}$ can be expressed as an average of~$n$ non-negative locally dependent random variables,
each bounded above by~$2m-1$, with dependence neighbourhoods of size $3m-2$; its $2k$-norm can then be bounded
using~\Ref{eq:locdepmombd} by $5em^{3/2}\sqrt{k/n}$, if $6km \le n$.  Thus
\[
     (1 - E - h_k)^{-k} \ \approx\ 1 \qquad\mbox{if} \quad k^{5/2} \ \ll\ n^{1/2}\IE D,
\]
so that the bound in~\eq{ADB-new-8} can be used as long as $y \ll n^{1/10}$, if~$\IE D$ is bounded
away from zero (or, equivalently, if~$p$ is bounded away from~$1$).  This yields
a leading term in the exponent of $-t^2/(2\mu\IE D)$, for $t \ll \s n^{1/10}$, rather than
$-t^2/(2\mu (2m-1))$, with 
\[
    1 - p^m \Le \IE D \Eq 1 + 2\,\frac{p-p^m}{1-p} - (2m-1)p^m \Le 7(2m-1)/12,
\]
for all $p$ and $m\ge2$, and indeed with $8/9 \le \IE D \le 2$ for all $p \le 1/3$ and $m \ge 2$.
This represents a considerable improvement over~\Ref{ADB-AB-bound} in the given range of~$t$.

Thus we see that our approach can compete with best known bounds, in situations where
other bounds are available.  However,
we emphasize again that the main advantage of our approach is that it gives bounds
in many settings where other methods cannot be applied.
For size--bias couplings in which the random variable~$|W^s-W|$ is not uniformly bounded, for instance,
exponential concentration inequalities are difficult to come by; see \cite{Ghosh2011}.

\subsection{Local neighbourhood statistics of \ER\ random graphs}\label{sec:localstater}

Let $\IG_n$ be the set of simple and undirected graphs on $n$ vertices with labels $[n]:=\{1,\ldots,n\}$, 
and let $\cG_n$ be an \ER\ random graph on $\IG_n$, with edge probability $p:=\lambda/n$. Fix $r\in\IN$ and,
for $G\in\IG_n$ and each $i=1,\ldots,n$, let $\cN_r(i,G)$  be the ``$r$-neighbourhood''
consisting of the subgraph induced by all vertices  
at a distance no greater than~$r$ from vertex~$i$ in~$G$, with vertex~$i$ distinguished; note that this includes 
the edges between vertices at graph distance $r$ from vertex~$i$.
For each $i=1,\ldots,n$, let $U$ be a real-valued function on graphs on at most~$n$ vertices having a distinguished vertex, 
% where the vertex labels are subsets of~$[n]$ that contain the label~$i$, 
and set
\be{
   X_i \Def U\bclr{\cN_r(i,\cG_n)};\qquad W \Def \sum_{i=1}^n X_i.
}
Concrete examples are given by taking $U$ to be any function of the degree of 
the distinguished vertex, with $r=1$, or the number of copies of some fixed subgraph~$H$ containing the distinguished vertex, with $r=\mathrm{diameter}(H)$. In the latter case, $W$ equals the number of occurrences of $H$ in the graph, times the number of vertices of $H$.

For sparse \ER\ random graphs, $r$-neighbourhoods are small with high probability. Hence, as long as~$U$ 
is well behaved, $\norm{D}_{2k}$ should be of a good order. We illustrate this principle in the following 
theorem.
%Note that in the application, we are able to bound central moments without computing the mean.
For a graph $G$, let $V(G)$ denote its vertex set, and $\abs{V(G)}$ the number of its vertices.

\begin{theorem}\label{thm:nhoodfuncmombd}
Fix $r\in\IN$ and $\b \ge 0$, and
let~$U$ be a function on graphs with a distinguished vertex, as above,
such that, for some positive constant $c$
and any graph~$G$ with a distinguished vertex in its domain,
\ben{\label{eq:uibd}
    \abs{U(G)}\Le  c \abs{V(G)}^\b.
}
Fix $\lambda>0$ and let $\cG_n$ be an \ER\ random graph on $\IG_n$, with edge probability $p:=\lambda/n$, and define
\be{
     X_i \Def U\bclr{\cN_r(i,\cG_n)}
}
and $W=\sum_{i=1}^n X_i$.
Then 
\be{
  n^{-1/2} \norm{W-\IE W}_{\rrr}\Le  c C(r,\b) \max\{\l,q(1+\beta)\}^{(1+2\b)r + 1/2},
}
where
\be{
C(r,\beta)=\sqrt{2\bclr{20^{1+\beta}+4^{1+\beta}}}\left(\frac{\pi e^{e-2}}{\log(e-1)}\right)^{(1+2\beta r)}.
}
\end{theorem}

\noindent For some basic examples, if $U$ is the indicator that the degree of the distinguished vertex~$i$ is in some set, 
we set $r=1$ and  can take $c=1$ and~$\b=0$ in~\eq{eq:uibd}.  If~$U$ is the number
of copies of a fixed subgraph~$H$ containing~$i$, we set $r=\mathrm{diameter}(H)$ and can take $c = \abs{V(H)}$ and $\b =\abs{V(H)}-1$.
These examples are standard, and may also be handled by other methods.  However, the theorem applies to more exotic 
statistics, such as the number of vertices having at least degree~$d$ with at most~$k$ neighbours having degree 
no greater than~$d$ ($r=2$, $c=1$, $\beta=0$). 

If~$X$ is a random variable such that $\|X-\IE X\|_{k} \le (Ck)^\a$ for all $k \ge k_0$, then
it follows from Markov's inequality that
\[
   \IP[|X-\IE X| \ge t] \Le \{(Ck)^\a/t\}^k,\quad k\ge k_0.
\]
Taking $k = C^{-1}(t/e)^{1/\a}$, this gives
\[
    \IP[|X - \IE X| \ge t] \Le e^{-C^{-1}(t/e)^{1/\a}}, \quad  t \ge e(Ck_0)^\a.
\]
Here, we take $X := W/\{cC(r,\b)\sqrt n\}$, $C := (1+\b)$, $k_0 := \l/(1+\b)$ and $\a := (1+2\b)r + 1/2$.
In particular, if $n\to\infty$ and $r$, $U$ and~$\l$ remain the same, then $W_n - \IE W_n$
is weakly concentrated on the scale $\sqrt n \{\log n\}^{(1+2\b)r + 1/2}$.

%\ignore{
%Then there is a Stein coupling such that the bounds on $\norm{W-\IE W}_{2k}$ of Theorems~\ref{thm:stnmomineq} 
%and~\ref{thm:stnmomineq0} apply, with $\eps=\eps'=T=0$,
%\be{
%     \norm{G}_{2k} \Le  n c \bclr{A(\lambda, 2k)^r+ A(\lambda, 1)^r},
%}
%and
%\bes{
%\norm{D}&_{2k}\Le  
%c\bbbclc{ \bclr{\pi e^{e-2}}^{1/(2k)}\bbbclr{\bbcls{\bclr{\tsfrac{2k}{\log(e-1)}}+
%	\bclr{\tsfrac{4k}{\log(e-1)}}^2}A(\lambda,2k)^{2 r}+A(\lambda,4k)^{4r}+A(\lambda,6k)^{6r} \\
%	&\qquad\quad+\bbcls{\bclr{\tsfrac{4k}{\log(e-1)}}+A(\lambda,4k)^{2r}}^2A(\lambda,4k)^{r-1}+2\bbcls{\bclr{\tsfrac{2k}{\log(e-1)}}+A(\lambda,2k)^{2r}} A(\lambda,2k)^{r-1}} \\
%	&\qquad\qquad\qquad\qquad+2A(\lambda,4k)^{6r}},
%}
%where 
%\be{
%    A(x,\ell) \Def \pi e^{e-2}\times
%\begin{cases}
%  \ell/\log\bclr{(e-1)\frac{\ell}{x}}, & \ell>x, \\
%     x, &\ell\leq x.
%    \end{cases}
% }   
% In particular,
%}

\section{Proofs}\label{sec:proofs}

In this section we prove the previous results.

\subsection{Proofs for general inequalities}

 \begin{proof}[Proof of Theorem~\ref{thm:stnmomineq}]
For $f(w)=(w-\mu)^{2k-1}$, using~\eq{eq:stncoup}, we have
\be{
   \IE[ G\{f(D+W)-f(W)\}] \Eq \IE \left[\left(W-\mu\right)^{2k}\right] + \IE \left[R \left(W-\mu\right)^{2k-1}\right].
}
%\ignore{
%Now using the assumption~\eq{eq:remcond} and the following inequalities from Jensen's (twice),
%\be{
%  \sigma\IE\babs{(W-\mu)^{2k-1}}\Le  \IE\bcls{(W-\mu)^2}^{1/2} \IE\bcls{(W-\mu)^{2k}}^{1-1/(2k)}\Le  \IE\bcls{(W-\mu)^{2k}},
%}
%}
Hence
\be{
   (1-\eps)\IE \left[\left(W-\mu\right)^{2k}\right]
      \Le \abs{\IE[ G\{f(D+W)-f(W)\}]} + \IE \left[T \babs{W-\mu}^{2k-1}\right].
}
%\note{
%Assuming $\norm{T}_{2k}\leq \sigma \eps'$, we can further upper bound 
%$\IE \left[T \babs{W-\mu}^{2k-1}\right]\leq \eps' \IE \left[\left(W-\mu\right)^{2k}\right]$, leading to the bound
%\be{
%(1-\eps-\eps')\IE \left[\left(W-\mu\right)^{2k}\right]
%      \Le \abs{\IE[ G\{f(D+W-\mu)-f(W-\mu)\}]} 
%}
%xxxx Maybe break out the last two displays as a lemma? Perhaps with the further decompositions in theorems below?xxx
%}
Furthermore, by the binomial theorem, we have
\be{
    f(D+W)-f(W) \Eq \sum_{j=1}^{2k-1} \binom{2k-1}{j} D^j (W-\mu)^{2k-1-j}.
}
Now, setting $x:=\|W-\mu\|_{2k}^2$ and using the triangle inequality, it follows that
\besn{\label{eq:key}
     (1-\eps)x^k \Le \sum_{j=1}^{2k-1}\binom{2k-1}{j} 
         \IE\left[\left|G D^j (W-\mu)^{2k-j-1}\right|\right] +  \IE \left[T \babs{W-\mu}^{2k-1}\right].
}
Using H\"older's inequality, and because $\|G\|_{2k} \le \CCC$ and $\|D\|_{2k} \le B$, this gives
\bes{
     (1-\eps) x^k &\Le \CCC \sum_{j=1}^{2k-1} \binom{2k-1}{j} B^{j} x^{\frac{2k-j-1}{2}} + 
                 \|T\|_{2k}x^{k-1/2} \\
     &\Eq x^{k-\frac{1}{2}}\left[\CCC \sum_{j=1}^{2k-1} \binom{2k-1}{j} \left(\frac{B}{\sqrt{x}}\right)^{j} 
                    + \|T\|_{2k} \right].
}
Rearranging, we thus have
\besn{\label{1}
    \sqrt{x} &\Le \frac{\CCC}{1-\eps} \sum_{j=1}^{2k-1} \binom{2k-1}{j} \left(\frac{B}{\sqrt{x}}\right)^{j} 
                  + \frac{\|T\|_{2k}}{1-\eps} \\
      &\Eq \frac{\CCC}{1-\eps} \left[\left(1+\left(\frac{B}{\sqrt{x}}\right)\right)^{2k-1}-1\right] 
                      + \frac{\|T\|_{2k}}{1-\eps}.
}
The solutions~$x$ to~\eq{1} that also satisfy $x> (2k-1) \CCC B/(1-\eps)$ are seen, by substituting  this bound into 
the right hand side of~\eq{1}, to be such that
\ben{\label{1.5}
   \sqrt{x} \Le \frac{\CCC}{1-\eps} \left[\left(1+\sqrt{\frac{B(1-\eps)}{\CCC(2k-1)}}\right)^{2k-1}-1\right]
               + \frac{\|T\|_{2k}}{1-\eps}.
}
Since, expanding the power, inequality~\eq{1.5} is satisfied by all $x \le (2k-1) \CCC B/(1-\eps)$, it follows
that all solutions of~\eq{1} satisfy~\eq{1.5}, 
proving the first inequality. The second follows from the fact that, for $t\geq 0$ and $\gamma\geq0$, 
$(1+t)^\gamma-1\leq \gamma t e^{\gamma t}$.
\end{proof}

\begin{proof}[Proof of Theorem~\ref{thm:stnmomineq0}]
First, take $r=2k$ with $k\in\IN$.  Then, using~\eq{eq:stncoup} with $f(w)=(w-\mu)^{2k-1}$, we have
\be{
   \IE[ G\{f(D+W)-f(W)\}] \Eq \IE \left(W-\mu\right)^{2k} + \IE \clc{R \left(W-\mu\right)^{2k-1}}.
}
Invoking the assumption~\eq{eq:remcond} and H\"older's inequality thus leads to
\bes{
   (1-\eps)\IE \left(W-\mu\right)^{2k}
      &\Le \abs{\IE[ G\{f(D+W)-f(W)\}]} + \IE \bclc{T \babs{W-\mu}^{2k-1}} \\
      &\Le \abs{\IE[ G\{f(D+W)-f(W)\}]} + \s^{-1}\|T\|_{2k}\clc{\s\|W-\mu\|_{2k}^{2k-1}} .
}
Now, for this choice of~$f$, using the fundamental theorem of calculus,  we have
\bes{
     \abs{f(D+W)-f(W)} &\Le (2k-1)|D|\max\{(W-\mu)^{2k-2},(W'-\mu)^{2k-2}\} \\
                               &\Le (2k-1)|D|\{(W-\mu)^{2k-2} + (W'-\mu)^{2k-2}\}.
}
Hence, by H\"older's inequality, and because $\IE\clc{(W'-\mu)^{2k}} \leq \IE\clc{(W-\mu)^{2k}}$,
it follows that
\bes{
    (1-\eps-\eps')\|W-\mu\|_{2k}^{2k} 
              &\Le (2k-1)\|G\|_{2k}\|D\|_{2k}\{ \|W-\mu\|_{2k}^{2k-2} + \|W'-\mu\|_{2k}^{2k-2}\}\\
                       &\Le 2(2k-1)\|G\|_{2k}\|D\|_{2k} \|W-\mu\|_{2k}^{2k-2},
}
and the theorem is proved for $r=2k$.

For $r=2k+1$, take $f(w) = (w-\mu)^{2k}\sgn(w-\mu)$, and observe that now
\bes{
     \abs{f(D+W)-f(W)} &\Le 2k|D|\max\{|W-\mu|^{2k-1},|W'-\mu|^{2k-1}\} \\
                               &\Le 2k|D|\{|W-\mu|^{2k-1} + |W'-\mu|^{2k-1}\};
}
the remainder of the argument is the same.
\end{proof}

\begin{proof}[Proof of Theorem~\ref{thm:stnmomineq2}]
 Much as in the proof of Theorem~\ref{thm:stnmomineq}, we begin by observing that
\bes{
  \lefteqn{ \IE \bclc{(W-\mu)^{2k}} + \IE \bclc{R \left(W-\mu\right)^{2k-1}} }\\
   &\Eq \sum_{j=1}^{2k-1} \binom{2k-1}{j}  \IE\bclc{GD^j(W-\mu)^{2k-1-j}} \\
   &\Eq (2k-1)\s^2 \IE\bclc{(W-\mu)^{2k-2}} + (2k-1)\IE\bclc{\bclr{\IE\clr{GD \giv W} - \s^2}(W-\mu)^{2k-2}} \\
   &\hskip1in + \sum_{j=2}^{2k-1} \binom{2k-1}{j}  \IE\bclc{GD^j(W-\mu)^{2k-1-j}}.
}
Hence, using the extra assumption, we have
\besn{\label{ADB-1}
%   \lefteqn{ 
   (1 - \eps_1-&\eps_2 - (2k-1)\eps_3)\IE \bclc{\left(W-\mu\right)^{2k}} 
   %}
   \\
        &\qquad \Le (2k-1)\s^2 \IE\bclc{(W-\mu)^{2k-2}} 
                  + (2k-1)\IE \bclc{T_2 \clc{W-\mu}^{2k-2}} \\
        &\hskip1in \mbox{} + \sum_{j=2}^{2k-1} \binom{2k-1}{j}  \IE\babs{GD^j(W-\mu)^{2k-1-j}}.
}
Defining $x := \|W-\mu\|_{2k}^2 \ge \s^2$, and bounding the final sum as in the proof of Theorem~\ref{thm:stnmomineq}, we deduce that
\bes{
    (1 - \eps_1 -\eps_2- (2k-1)\eps_3)x^k &\Le
          (2k-1)\s^2 x^{k-1} + (2k-1)\|T_2\|_k x^{k-1} \\
      &\hskip2cm\mbox{} +  \CCC x^{k-1/2}\bbbclc{\bbclr{1 + \frac B{\sqrt{x}}}^{2k-1}
               - 1 - \frac{(2k-1)B}{\sqrt{x}}} .
}
Using the fact that, for $a \ge 2$ and $y > 0$, we have
\ben{\label{eq:2tay}
    (1+y)^a - 1 -ay \Le \half a(a-1)y^2(1+y)^{a-2} \Le \half a(a-1)y^2 e^{ay},
}
it follows that
\bes{
   \lefteqn{ \bclr{1 - \eps_1-\eps_2 - (2k-1)\eps_3} x }\\
     &\qquad\Le (2k-1)\bbbclc{\s^2\bbbclr{1 + \frac{(k-1)\CCC B^2}{\s^2 \sqrt x} \exp\bbclc{\frac{(2k-1)B}{\sqrt x}}} 
                       + \|T_2\|_k},
}
from which, considering first the case $x \ge (2k-1)\s^2$, it follows that
\bes{
   x &\Le \frac{(2k-1)\s^2}{1-\eps_1-\eps_2-(2k-1)\eps_3}
      \bbbclr{1 + \frac{(k-1)\CCC B^2}{\s^3\sqrt{2k-1}} \exp\bbclc{\frac{B\sqrt{2k-1}}{\s}} + \frac{\|T_2\|_k}{\s^2}},
}
as claimed.
\end{proof}

\begin{proof}[Proof of Theorem~\ref{thm:stnmomineq3}]
We begin from~\eq{ADB-1}. Defining $x_k := \|W-\mu\|_{2k}^2\geq \sigma^2$,  we now deduce that
\besn{\label{ADB-2}
    (1 - E)x_k^k &\Le
          (2k-1)\s^2 x_{k-1}^{k-1} +  \CCC x_k^{k-1/2}\bbbclc{\bbclr{1 + \frac B{\sqrt{x_k}}}^{2k-1}
               - 1 - \frac{(2k-1)B}{\sqrt{x_k}}}.
}
Using the Taylor expansion argument~\eq{eq:2tay}, we have
\bes{
   \CCC x^{k-1/2}\bbbclc{\bbclr{1 + \frac B{\sqrt{x}}}^{2k-1}
               - 1 - \frac{(2k-1)B}{\sqrt{x}}} &\Le x^k H_k(x),
}
where
\[
     H_k(x)  \Def x^{-3/2}(2k-1)(k-1)\CCC B^2 \exp\clc{(2k-1)B/\sqrt x}. 
\]
Because $\s \ge B\sqrt{e(2k-1)}$, and from Lemma~\ref{lem:asymnormoms}, we have
\[
     H_k(\s^2 \|N\|^2_{2k}) \Le  h_k.
\]
 It thus follows from~\eq{ADB-2} that either $\sqrt{x_k} \le \s\|N\|_{2k}$ or
\[
    (1-E-h_k)x_k^k \Le (2k-1)\s^2 x_{k-1}^{k-1}.
\]
Using the same argument, and because $h_k$ is increasing in~$k$, we deduce that, for each $2 \le l\le k$,
\[
    x_l^l \Le \max\bbclc{\frac{(2l-1)\s^2}{1-E-h_k} x_{l-1}^{l-1},\clr{\s\|N\|_{2l}}^{2l}}.
\]
Iterating this inequality, starting with $l=k$ and working downwards, and noting 
that $\|N\|_{2l}^{2l}=(2l-1)\|N\|_{2(l-1)}^{2(l-1)}$, we obtain
\[
     x_k^k \Le (1-E-h_k)^{-k}\clr{\s\|N\|_{2k}}^{2k},
\]
proving the theorem.
\end{proof}

\subsection{Proofs for concentration inequalities}
\begin{proof}[Proof of Corollary~\ref{ADB-bounded}]
If $|D| \le x_1$ and $|n^{-1}G| \le x_2$ a.s.\ for all~$n$, in order to exploit~\eq{ADB-new-0}, we start by minimizing
\[
     \log\bbbclc{t^{-2k}\bbclc{2knx_1x_2}^k} \Eq -2k \log t + k\{\log n + \log (2k) + \log (x_1x_2)\}
\]
in~$k$.  The calculus minimum is attained when
\[
    \log n + \log (2k) + \log (x_1x_2) - 2\log t + 1 \Eq 0.
\]
Thus, choosing $k := k_n(t) := \bfloor{\half\{t^2/(nx_1x_2e)\}}$, this gives
\[
     (2k)^k \Le \bbbclc{\frac{t^2}{nx_1x_2e}}^k,
\]
and hence
\[
   t^{-2k}\bbclc{n(2k-1)x_1x_2}^k \Le e^{-k} \Le e\exp\{-\half t^2/(nx_1x_2e)\}.
\]
%because $\half\{t^2/(nx_1x_2e)\} \ge 1$. 
Substituting this into~\eq{ADB-new-0} gives
\ben{\label{ADB-new-5}
   \IP[|W_n - \mu_n| > t] \Le e\exp\bbbclc{-\frac{t^2}{2nx_1x_2e}
           \bbbclr{1 - \frac{2}{\sqrt n}\,\sqrt{\frac{(2k-1)x_1}{x_2}}}},
}
with $k$ as above. Since %$\half\{t^2/(nx_1x_2e)\} \ge 1$, we have
\be{
     2k-1 \Le 2k \Le \frac{t^2}{nx_1x_2 e},
}
\eq{ADB-new-5} is bounded as claimed.
\end{proof}

\begin{proof}[Proof of Corollary~\ref{ADB-Weibull-tails}]
\mbox{}
Considering~\eq{ADB-new-0}, we note that we can take 
\[
     \a_{n,k} \Eq \b_{n,k} \Eq \bbbclc{\frac{2kc}{a}\,\bbclr{b^{-1/a}}^{2k} \Gamma(2k/a)}^{1/2k},
\]
from Lemma~\ref{moments-from-tails}.
We then choose~$k$ to (approximately) minimize
\[
   -2k \log t + k\{\log n + \log (2k)\} + \log\bbbclc{\frac{2kc}{a}\,\bbclr{b^{-1/a}}^{2k} \Gamma(2k/a)},
\]
or, more simply, to (approximately) minimize, using a crude approximation of the gamma function along the lines of Stirling's formula,
\[
   -2k \log t + k\{\log n + \log (2k)\} - (2k/a)\log b + (2k/a)\log(2k/a) - 2k/a.
\]
The calculus minimum is attained when
\[
   \log n + \log (2k) +1 - 2\log t + (2/a)\{-\log b + \log(2k/a)\} \Eq 0,
\]
suggesting the choice, provided that $t^2 \ge ne 2^{(a+2)/a}(ab)^{-2/a}$ so that $k\geq1$, of
\ben{\label{ADB-new-5.4}
   k \Eq \bbbfloor{\half\bbbclr{\frac{t^2}{ne}}^{a/(a+2)}(ab)^{2/(a+2)}},
}
for which 
\ben{\label{ADB-new-5.5}
    (2k)^{k(a+2)/a} \Le \bbbclc{\frac{t^2}{ne}}^k (ab)^{2k/a}.
}
Then, using the bound on the  gamma function from \cite[Corollary~1.2]{Batir2008},  the right hand side of~\eq{ADB-new-0} is at most 
\bes{
   \lefteqn{  t^{-2k}n^k(2k)^{k} c\,\bbclr{b^{-1/a}}^{2k} \Gamma(2k/a+1) \exp\{n^{-1/2}(2k)^{3/2}\} }\\
    &\qquad\Le t^{-2k}n^k(2k)^{k} c\,\bbclr{b^{-1/a}}^{2k}\,\sqrt{2\pi} \,\bbbclr{\frac{2k}{a}}^{2k/a}   \bbbclc{\frac{2k}{a}+\frac{1}{2}}^{1/2}      e^{-2k/a}\,\exp\{n^{-1/2}(2k)^{3/2}\} \\
    &\qquad\Le c\sqrt{2\pi}\, \bbbclc{\frac{2k}{a}+\frac{1}{2}}^{1/2}  e^{-k(a+2)/a}\,\exp\{n^{-1/2}(2k)^{3/2}\},
}
because of~\eq{ADB-new-5.5}.
The result follows after noting that~$k$ from~\eq{ADB-new-5.4} equals $k_t$ from  the statement of the corollary.
\end{proof}

\begin{proof}[Proof of Corollary~\ref{cor:normal}]
Using the upper and lower factorial bounds given by \cite{Robbins1955}, we have
\[
    \|N\|_{2k}^{2k} \Eq \frac{(2k)!}{k!2^k} \Le \exp\{k\log k + (k+\half)\log 2 - k\},
\]
and so 
\[
     (2k)^{-k}\|N\|_{2k}^{2k} \Le \exp\{\half\log2 - k\} \Eq \sqrt 2\, e^{-k}.
\]
Let $k = \bfloor{y^2/2}$ and take $t = \s_n \sqrt{2k}$  in~\eq{eq:Mineq}; then, using Theorem~\ref{thm:stnmomineq3} and the previous display,
\[
   \IP[|W_n-\mu_n| > y\s_n] \Le \IP[|W_n-\mu_n| > \s_n\sqrt{2k}] \Le \frac{\sqrt2\, e^{-k}}{(1 - E - h_{k})^{k}} \Le 
         \frac{e\sqrt2\, e^{-y^2/2}}{(1 - E - h_{k})^{k}},
\]
proving the corollary.
\end{proof}

\subsection{Proofs for applications}

\paragraph{Independent sums}

\begin{proof}[Proof of Proposition~\ref{ADB-independent}]
As in the proof of Theorem~\ref{thm:stnmomineq} (but with $R=0$), we begin with
\besn{\label{ADB-3}
   \IE \bclc{W^{2k}} 
   &\Eq \sum_{j=1}^{2k-1} \binom{2k-1}{j} (-1)^j \IE\bclc{nX_I^{j+1} W^{2k-1-j}}.
}
For the first term in the sum, we have
\[
    (2k-1)\IE\bclc{nX_I^{2}W^{2k-2}} \Eq (2k-1)\sum_{i=1}^n \IE\bclc{X_i^2 W^{2k-2}}.
\]
By the independence of $X_i$ and~$W_i:=W-X_i$, we have
\[
   \IE\bclc{(X_i^2 - \s_i^2)W^{2k-2}} \Eq \IE\bclc{(X_i^2 - \s_i^2)\{W^{2k-2} - W_i^{2k-2}\}},
\]
and so, using the fundamental theorem of calculus, we have
\bes{
   \lefteqn{ \babs{\IE\bclc{(X_i^2 - \s_i^2)W^{2k-2}}} }\\
      &\qquad\Le \IE\bclc{|X_i^2-\s_i^2||X_i|(2k-2)\{|W|^{2k-3} + |W_i|^{2k-3}\}}.
}
Now, Jensen's inequality implies 
\be{
 \IE |W_i|^{2k}= \IE \bbcls{\babs{W_i + \IE[X_i | W_i]}^{2k}} \leq 
\IE \bbcls{ \IE\bcls{\abs{W_i + X_i}^{2k} \vert W_i }}=\IE\bcls{\abs{W}^{2k}},
}
and using this with H\"older's inequality yields
\[
    \IE\bclc{|X_i^3|\{|W|^{2k-3} + |W_i|^{2k-3}\}} \Le 2\|X_i\|_{2k}^3 \|W\|_{2k}^{2k-3},
\]
and $\s_i^2\IE\bclc{|X_i|\{|W|^{2k-3} + |W_i|^{2k-3}\}}$ is bounded by the same quantity.
Hence
\bes{
    \lefteqn{ (2k-1)\babs{\IE\bclc{nX_I^{2}W^{2k-2}} - \s^2\IE\bclc{W^{2k-2}}} }\\
            &\qquad\Le 8(k-1)(2k-1)\bbclc{\sum_{i=1}^n \|X_i\|_{2k}^3} \|W\|_{2k}^{2k-3} \\
            &\qquad\Le 8n(k-1)(2k-1) \|X_I\|_{2k}^3 \|W\|_{2k}^{2k-3}.
}
The remaining terms in the sum in~\eq{ADB-3} are equal to
\bes{
 nX_I\bcls{(W-X_I)^{2k-1} - W^{2k-1} + (2k-1) X_I W^{2k-2} },
 }
 which can be bounded in absolute value, using Taylor's theorem with integral remainder, by  
 \bes{ 
  \lefteqn{ (2k-1)(k-1)\IE\bclc{n|X_I|^3 \{|W|^{2k-3} + |W'|^{2k-3}\}} }\\
          &\qquad\Le 2n(k-1)(2k-1) \|X_I\|_{2k}^3 \|W\|_{2k}^{2k-3}.
}
We have thus shown that
\be{
\norm{W}_{2k}^{2k}\Le  n(2k-1)\sigma^2 \norm{W}_{2k-2}^{2k-2}+
\|W\|_{2k}^{2k}\left( 10n^{-1/2}(k-1)(2k-1) \rho_{k}^3 \frac{\sigma^3}{\|W\|_{2k}^{3}}\right).
}
So if $\norm{W}_{2k}\geq \sigma \norm{N}_{2k}$, then
\be{
    10n^{-1/2}(k-1)(2k-1) \rho_{k}^3 \frac{\sigma^3}{\|W\|_{2k}^{3}}
       \Le 5\sqrt{2e^3}\rho_k^3 \sqrt{\frac{k-1}n} \Eq h_k'.
}
Arguing as in the proof of Theorem~\ref{thm:stnmomineq3} completes the proof.
\end{proof}

\paragraph{Sums of local statistics of sparse \ER\ random graphs}
\begin{proof}[Proof of Theorem~\ref{thm:nhoodfuncmombd}]
%Let the set of vertices of~$\cG_n$ be denoted by $[n] := \{1,2,\ldots,n\}$.
To obtain moment bounds for~$W$, we define a Stein coupling. For each $j=1,\ldots,n$,
set
$$
      \II_r\sp{j} \Def \bclc{\{i,l\} \subset [n]\colon\, \{i,l\} \cap \cN_r(j) \neq \emptyset}.
$$
Let $(E'_{il}, \{i,l\} \subset [n])$ be independent indicators, independent
also of~$\cG_n$, each with $\IE E'_{il} = p$.
Given~$\cG_n$, define the random graph~$\cG_n\sp{j}$ by replacing all the 
edge indicators $(E_{il},\, \{i,l\} \in \II_r\sp{j})$, between pairs of vertices with at least one vertex 
in the $r$-neighbourhood of~$j$, with $(E'_{il},\, \{i,l\} \in \II_r\sp{j})$,
leaving the edges between other pairs of vertices the same as those of $\cG_n$.
Define
\be{
      X_i\sp{j} \Def U\bclr{\cN_r(i,\cG_n\sp{j})},
}
and $W\sp{j}:=\sum_{i=1}^n X_i\sp{j}$.
Finally, let $J$
be uniform on the set~$\{1,\ldots,n\}$, independent of the random objects above, and define
$W':=W\sp{J}$ and $G=-n\clr{ X_J-\IE X_J}$. After noting that $X_j$ is independent of 
$W\sp j$, it is easy to see that $(W,W',G)$ is an exact Stein coupling.
Moreover, $\law(W')=\law(W)$, so that the central moments of $W$ and~$W'$ are equal, and hence both
of Theorems~\ref{thm:stnmomineq} and~\ref{thm:stnmomineq0}
apply, with $\eps = \eps' = 0$; we use Theorem~\ref{thm:stnmomineq0}.

To apply it to bound the $\rrr$-th central moment, we need to bound
\ban{
   \norm{G}_{\rrr}^{\rrr}&\Eq n^{\rrr-1}\sum_{i=1}^n \IE|X_i-\IE X_i|^{\rrr},
%=n^{r}\IE\bcls{(X_1-\IE X_1)^{r}}
 \label{eq:intt1a}
}
and
\ban{
     \norm{D}_{\rrr}^{\rrr} &\Eq \frac{1}{n}\sum_{j=1}^n \IE\babs{W\sp{j}-W}^{\rrr} \notag\\
     &\Eq \frac{1}{n}\sum_{j=1}^n \IE\bbbabs{\sum_{i=1}^n (X_i\sp{j}-X_i)}^{\rrr}
%     \Eq  \IE\bbbabs{\sum_{i=1}^n (X_i\sp{1}-X_i)}^{\rrr}.
%\notag\\
%	&=\IE\bbbcls{\bbclr{\sum_{i=1}^n (X_i\sp{1}-X_i)}^{2k}}, 
\label{eq:intt1b}
}
% \blue{We assume in what follows that $X_i \ge 0$;  in the general case, consider
% $\max\{X_i,0\}$ and $-\min\{X_1,0\}$ separately. }
% where the final equalities are by symmetry.

To bound~\eq{eq:intt1a}, define $N_r(i):=\abs{V(\cN_r(i,\cG_n))}$ to be 
the number of vertices in the $r$-neighbourhood of $i$ in $\cG_n$, and 
note that, due to~\eq{eq:uibd},
\be{
    |X_i| \Le cN_r(i)^\b,
} 
and in particular, from Lemma~\ref{lem:nhoodmombd},
for 
\be{
    A(x,\ell) \Def \pi e^{e-2}\times
\begin{cases}
  \ell/\log\bclr{(e-1)}, & \ell>x, \\
     x, &\ell\leq x,
    \end{cases}
 } 
 we have
\ben{\label{eq:mnbd}
   \norm{X_i}_{\ell}\Le  c \norm{N_r(i)^\b}_{\ell} 
       \Eq  c\norm{N_r(i)}_{\b\ell}^{\b} \Le  2^\b c A(\lambda, \ell\b)^{r\b}.
}
Using~\eq{eq:mnbd} and Minkowski's inequality, this gives
\be{
   \IE\babs{X_i-\IE X_i}^{q} \Eq \norm{X_i-\IE X_i}_{\rrr}^{\rrr}
 		\Le \bclr{\norm{X_i}_{\rrr}+\norm{X_i}_1}^{\rrr}
		\Le (2^{1+\b} c)^{\rrr} \, A(\lambda, \rrr\b)^{r\b\rrr}.
}
Thus, from~\eq{eq:intt1a}, we can bound
\ben{ \label{G-norm-bnd}
    \norm{G}_{\rrr}\Le  2^{1+\b} n c (C_A \L_{\rrr}\sp{1})^{r\b},
}
where $\L_{s}\sp{1} := \max\{\l,s\b\}$ and $C_A=\pi e^{e-2}/\log(e-1)$.

Moving to~\eq{eq:intt1b}, we note that
$z_j := \IE\babs{\sum_{i=1}^n (X_i\sp{j}-X_i)}^{\rrr}$ is the same for all~$j$, 
so that $\norm{D}_q^q$ in~\eq{eq:intt1b} can be bounded by~$z_1$.
Define the indicators
\bes{
%           I_i &\Def I[E_{il} \Eq  1 \mbox{ for some } l \in \cN_r(1,\cG_n)];\\
           I_i &\Def I[E'_{il} \Eq  1 \mbox{ for some } l \in V(\cN_r(1,\cG_n))],
}
and then the subsets of vertices
\be{
  \II\ui_- \Def V\bclr{\cN_{r+1}(1, \cG_n)}; \qquad
  \II\ui_+ \Def \{i \notin V(\cN_r(1,\cG_n))\colon\, I_i=1\}.
}
The difference $X_j\ui - X_j$ can be written as $(X_j\ui - \hX_j\ui) + (\hX_j\ui - X_j)$,
where $\hX_j\ui := U\bclr{\cN_r(j,\cG\sp1_{n,r})}$, and $\cG\sp1_{n,r}$ is defined to be 
%the $s$-neighbourhood
%of~$j$ in $\cG\ui_r$, 
the subgraph of~$\cG_n$ consisting of the edges between the vertices 
$\{i \notin N_r(1,\cG_n)\}$; in particular, vertices in $N_r(1,\cG_n)$ are isolated in $\cG\sp1_{n,r}$.  
Then $\hX_j\ui - X_j$ can only be non-zero  if $j \in \cN_{r-1}(i, \cG_n)$ 
for some $i \in \II\ui_-$, and $X_j\ui - \hX_j\ui$ can only be non-zero if $j \in \cN_{r-1}(i, \cG_n)$
for some $i \in V\bclr{\cN_r(1,\cG_n)}\cup\II\ui_+$. In the former case, 
\[
    |\hX_j\ui - X_j| \Le 2cN_r(j,\cG_n)^\b \Le 2cN_{3r}(1,\cG_n)^\b,
\]
and there are at most $N_{2r}(1,\cG_n)$ such indices;  hence
\ben{\label{ADB-lower-bnd}
    \sum_{j=1}^n |\hX_j\ui - X_j| \Le 2c N_{2r}(1,\cG_n)\{N_{3r}(1,\cG_n)\}^{\b}. 
}

To bound $|X_j\ui - \hX_j\ui|$, define
\bes{ 
   \cN^*_{lr}(1) &\Def \cN_r(1,\cG_n) \cup \hNN_{lr}(1),
}
where
\bes{
     \hNN_{lr}(1) \Def  \bigcup_{i \in \II\ui_+} \cN_{(l-1)r-1}(i,\cG_{n,r}\sp1) ,\qquad l=2,3;
}
set $\hN_{lr}(1) := |V(\hNN_{lr}(1))|$, and note that
$N_{lr}^*(1) := |V(\cN^*_{lr}(1))| \le N_r(1) + \hN_{lr}(1)$.  Then there are at most $N^*_{2r}(1)$ indices~$j$
at which $|X_j\ui - \hX_j\ui|$ can be non-zero, and, for such~$j$, the difference cannot exceed
\[
   c\babs{V\bclr{\cN_r(j,G_{n}\sp{1})}}^\b + c\{N_r(j)\}^\b   \Le 2c\{N^*_{3r}(1)\}^\b.
\] 
Hence
\ben{ \label{ADB-upper-bnd}
    \sum_{j=1}^n |X_j\ui - \hX\ui_j| \Le 2c N^*_{2r}(1)\{N^*_{3r}(1)\}^{\b}.
}

Now, by H\"older's inequality,
\[
   \|N_{2r}(1)\{N_{3r}(1)\}^{\b}\|_{\rrr} \Le \|N_{2r}(1)\|_{(1+\b)\rrr}\,\|\{N_{3r}(1)\}^{\b}\|_{(1+\b)\rrr/\b}
       \Eq \|N_{2r}(1)\|_{(1+\b)\rrr}\,\|N_{3r}(1)\|_{(1+\b)\rrr}^\b,
\]
which can be bounded using Lemma~\ref{lem:nhoodmombd}; this gives
\ben{ \label{ADB-1st-bnd}
     \|N_{2r}(1)\{N_{3r}(1)\}^{\b}\|_{\rrr} 
      \Le 2^{1+\b} A(\lambda,(1+\b)\rrr)^{2r+3r\b} \Le 2^{1+\b} \{C_A \L_{\rrr}\sp{2}\}^{2r+3r\b},
}
where $\L_{s}\sp{2} := \max\{\l,s(1+\b)\}$.  Similarly, we have
\ben{\label{ADB-hat-norms}
    \|N^*_{ {2}r}(1)\{N^*_{ {3}r}(1)\}^{\b}\|_{\rrr} \Le \|N^*_{ {2}r}(1)\|_{(1+\b)\rrr}\,\|N^*_{ {3}r}(1)\|_{(1+\b)\rrr}^\b,
}
where $N^*_s(1) \le N_r(1) + \hN_s(1)$.  The norms of~$N_r(1)$ are bounded once more by Lemma~\ref{lem:nhoodmombd}. 
For $\hN_{lr}(1)$, $l= {2,3}$, 
we set $N_{(l-1)r-1}\ui(i):=\abs{V(\cN_{(l-1)r-1}(i,\cG_{n,r}\ui))}$ and observe that, conditional on~$\cN_r(1,\cG_n)$, 
the indicators $\{I_i,\,i\notin \cN_r(1,\cG_n)\}$ are independent of the quantities
$\{N_{(l-1)r-1}\ui(i),\,i \notin \cN_r(1,\cG_n)\}$;  the $I_i$ all have expectation less
than $pN_r(1)$, and, for $i \notin \cN_r(1,\cG_n)$, each $N_{(l-1)r-1}\sp{1}(i)$
 has distribution dominated by the
{\em unconditional\/} distribution of $N_{(l-1)r-1}(i)$, because~$\cN_{(l-1)r-1}(i,\cG_{n,r}\sp{1})$ is constructed
just as~$\cN_{(l-1)r-1}(i,\cG_n)$, but from a subset of the vertices~$[n]$.  Applying Lemma~\ref{lem:condmink},
with $\cE := \II\ui_+$ and $Y_i := N\ui_{(l-1)r-1}(i)$, it follows that
\[
    \IE\bbclr{\hN_{lr}(1)^{s} \giv \hN_r(1)}  
       \Le  \IE\bclc{\hB^{s} \giv N_r(1)}\,\IE\bbclc{\bclr{N_{(l-1)r-1}(i)}^{s}},
\]
where $\hB \sim \Bi(n,pN_r(1))$.  Bounding the expectations by Lemmas \ref{lem:binmombd} and~\ref{lem:nhoodmombd},
this in turn gives
\bes{
   \IE\bbclr{\hN_{lr}(1)^{s} \giv N_r(1)} 
       &\Le  A(\lambda N_r(1), s)^{s}\{2 A(\lambda,s)^{(l-1)r-1}\}^{s} \\
       &\Le  \{C_A (\lambda N_r(1) + s)\}^{s}\{2 A(\lambda,s)^{(l-1)r-1}\}^{s}.
}
Invoking Lemma~\ref{lem:nhoodmombd} once again, it thus follows that
\bes{
   \|\hN_{lr}(1)\|_{s} &\Le  C_A
         \bclc{\l\|N_r(1)\|_{s} + s} \,\{2 A(\lambda,s)^{(l-1)r-1}\}   \\
        &\Le   C_A \bclc{2\l A(\l,s)^{r} + s} \,\{2 A(\lambda,s)^{(l-1)r-1}\} \\
        &\Le   8\{C_A(\l\vee s)\}^{lr}.
}
Altogether we find that, for $s\geq1$ and $l=2,3$, 
\be{
     \|N^*_{lr}(1)\|_{s}\leq \norm{N_r(1)}_s+\norm{\hN_{lr}(1)}_s 
          \leq 2\bclc{C_A (\lambda \vee s)}^r+8\{C_A(\l\vee s)\}^{lr}\leq 10\{C_A(\l\vee s)\}^{lr}.
}
Hence, from~\eq{ADB-hat-norms}, 
\ben{\label{ADB-2nd-bnd}
    \|N^*_{2r}(1)\{N^*_{3r}(1)\}^{\b}\|_{\rrr} 
%    \Le 10\{ C_A(\l\vee (1+\b)\rrr)\}^{2r}\,
%              10^\b \{C_A(\l\vee (1+\b)\rrr)\}^{3r\b}
               \Le 10^{1+\b}\{C_A \L_{\rrr}\sp{2}\}^{2r+3r\b}.
}
Combining \eq{ADB-1st-bnd} and~\eq{ADB-2nd-bnd} with \eq{eq:intt1b}, \eq{ADB-lower-bnd} and~\eq{ADB-upper-bnd},
we see that
\ben{ \label{D-norm-bnd}
     \|D\|_{\rrr} \Le (10^{1+\b} + 2^{1+\b})c (C_A \L_{\rrr}\sp{2})^{2r+3r\b}.
}

Now, applying Theorem~\ref{thm:stnmomineq0} with the bounds \eq{G-norm-bnd} and~\eq{D-norm-bnd}, the theorem
follows.
\end{proof}

\begin{lemma}\label{lem:condmink}
Let $\cI$ be a finite index set, $\cE$ be a random (possibly empty) subset of $\cI$,
and define $E:=\abs{\cE}$. Let $\{Y_i\}_{i\in\cI}$ be a collection of random variables
independent of $\cE$ and, for $\ell\in \IN$,  
let  $\max_{i\in\cI} \norm{Y_i}_\ell \le y$.
Then
\be{
   \bbnorm{\sum_{i\in\cE} Y_j}_{\ell}\Le  y\norm{E}_{\ell} .
}
\end{lemma}

\begin{proof}
Using Minkowski's inequality in the third line, we have
\ba{
\bbnorm{\sum_{i\in\cE} Y_i}_{\ell}^\ell&\Eq \IE\bbbcls{\bbclr{\sum_{i\in\cE} Y_i}^\ell} \\
		&\Eq \sum_{\cS\subseteq\cI} \IE\bbbcls{\bbclr{\sum_{i\in\cS} Y_i}^\ell \big| \cE=\cS} \IP(\cE=\cS)\\
		&\Le  \sum_{\cS\subseteq\cI} \bbclr{\sum_{i\in\cS} \norm{Y_i}_{\ell}}^\ell \IP(\cE=\cS)\\
		&\Le  y^\ell\sum_{\cS\subseteq\cI}^\infty \abs{\cS}^\ell   \IP(\cE=\cS)
                          \Eq  y^\ell\norm{E}_{\ell}^\ell. \qedhere
}
\end{proof}

\medskip
\begin{lemma}\label{lem:binmombd}
Let $n\in\IN$, $0\leq p\leq 1$, $Y\sim\Bi(n,p)$, and $\ell\in\IN$.
Then
$\norm{Y}_{\ell}\leq A(np,\ell)$,
where 
\be{
    A(x,\ell) \Def \pi e^{e-2}\times
\begin{cases}
  \ell/\log\bclr{(e-1)}, & \ell>x, \\
     x, &\ell\leq x.
    \end{cases}
 }   
In particular,  $A(x,l) \le C_A(x\vee l) \le C_A(x+l)$, where $C_A := \pi e^{e-2}/\log\clr{e-1}$.
\end{lemma}

\begin{proof}
For any $\theta>0$, we have, using Markov's inequality,
\ban{
\IE[Y^\ell]&\Eq  \ell \int_0^\infty t^{\ell-1} \IP(Y\geq t) dt\notag \\
	&\Le  \ell \int_0^\infty t^{\ell-1} \IE[e^{\theta Y}] e^{-\theta t} dt \notag \\
	&\Eq  \Gamma(\ell+1) \theta^{-\ell} \bclr{1-p+pe^{\theta}}^n. \label{eq:gammgfbd}
}
If $\ell>np$, we choose $\theta=\log\bclr{(e-1) \ell/(np)}$, which yields
\be{
\norm{Y}_{\ell}\Le  \Gamma(\ell+1)^{1/\ell} \frac{\bclr{1-p+\frac{(e-1) \ell}{n}}^{n/\ell}}{\log\bclr{(e-1)\ell/(np)}}\Le  \Gamma(\ell+1)^{1/\ell} \frac{e^{e-1}}{\log\bclr{(e-1) \ell/(np)}}.
}
Now using the explicit bound on the Gamma function 
given in \cite[Corollary~1.2]{Batir2008}:
\ben{\label{eq:Batirgambd}
\Gamma(\ell+1)\Le  \sqrt{\pi(2\ell+1)} \ell^\ell e^{-\ell},
}
we find 
\be{
\norm{Y}_{\ell}\Le  \frac{e^{e-2}\ell}{\log\bclr{\frac{(e-1)\ell}{np}}} \bclr{\pi(2\ell+1)}^{1/(2\ell)}\Le  \frac{e^{e-2} \pi \ell}{\log(e-1)}, %\Le   \frac{e^{e-2} \pi \ell}{\log\bclr{\frac{(e-1)\ell}{np}}} ,
}
where the second inequality follows from elementary considerations, noting $\ell>np$.

If $\ell\le np$, set $\theta=\ell/(np)$ in~\eq{eq:gammgfbd} and use~\eq{eq:Batirgambd} to find
\be{
\norm{Y}_{\ell}\Le  np  \frac{ \Gamma(\ell+1)^{1/\ell}}{\ell} \bclr{1-p+pe^{\ell/(np)}}^{n/\ell}
	\Le  np  \bclr{\pi(2\ell+1)}^{1/(2\ell)} e^{-1} e^{\frac{np}{\ell}\bclr{e^{\ell/(np)}-1}}.
}
Recalling that $\ell\leq np$, the final exponential is maximized at $\ell=np$, which 
leads to 
\be{
\norm{Y}_{\ell}\Le  np  \bclr{\pi(2\ell+1)}^{1/(2\ell)} e^{e-2}\Le   e^{e-2} \pi (np).
\qedhere
}
\end{proof}

\begin{lemma}\label{lem:nhoodmombd}
Fix $n,r\in\IN$ and $p\in[0,1]$, and let $N_r$ be the number of vertices in the $r$-neighbourhood
of a given vertex in an \ER\ random graph on $n$ vertices with edge probability parameter $p$. Then
\be{
     \norm{N_r}_{\ell}\Le  \frac{A(np,\ell)^{r+1}-1}{A(np,\ell)-1} \Le  2A(np,\ell)^{r},
}
where $A(x,\ell)$ is defined in Lemma~\ref{lem:binmombd}.
\end{lemma}

\begin{proof}
Let $(Z_0,Z_1,Z_2,\ldots)$ be the sequence of  generation sizes
for a Galton-Watson branching process  with offspring distribution $\Bi(n,p)$, started from a single individual; $Z_0=1$. 
Then, using the usual exploration process process coupling found, for example, in \cite[Chapter~4]{Hofstad2017}, $\sum_{s=0}^r Z_s$ stochastically dominates $N_r$, so that
 $\norm{N_r}_\ell \leq \sum_{s=0}^r \norm{Z_s}_\ell$, for any $r,\ell\geq0$. To bound $\norm{Z_s}_{\ell}$, we use the random 
sum representation of a branching process and repeatedly apply Lemma~\ref{lem:condmink} to find
\be{
   \norm{Z_s}_{\ell}\Le  \norm{B}_{\ell}^{s} ,
}
where $B\sim\Bi(n,p)$.
The result now easily follows from Lemma~\ref{lem:binmombd}, and because $A(x,l) \ge 2$ for all $x > 0$ and $l \in \IN$.
\end{proof}

\section*{Acknowledgments}
The authors were partially supported by the Australian Research Council Discovery Grant DP150101459 and 
the ARC Centre of Excellence for Mathematical and Statistical Frontiers, CE140100049. 
We thank the referee for their helpful comments, and Larry Goldstein and Peter Eichelsbacher for pointing out some errors and omissions in a previous draft.
%
%\bibliographystyle{mynatbib}
%\bibliography{/Users/rossn1/Desktop/Dropbox/Bib}

\end{document}